\documentclass[11pt,reqno,tbtags,a4paper]{amsart}
\usepackage{amssymb}
\usepackage{mathabx}  
\usepackage{xpunctuate}
\usepackage{url}
\usepackage[square,numbers]{natbib}
\bibpunct[, ]{[}{]}{;}{n}{,}{,}

\title
{On a central limit theorem in renewal theory}

\date{22 May, 2023}

\author{Svante Janson}
\thanks{Supported by the Knut and Alice Wallenberg Foundation}
\address{Department of Mathematics, Uppsala University, PO Box 480,
SE-751~06 Uppsala, Sweden}
\email{svante.janson@math.uu.se}
\newcommand\urladdrx[1]{{\urladdr{\def~{{\tiny$\sim$}}#1}}}
\urladdrx{http://www.math.uu.se/~svante/}

\subjclass[2020]{} 

\numberwithin{equation}{section}

\renewcommand\le{\leqslant}
\renewcommand\ge{\geqslant}

\allowdisplaybreaks

\theoremstyle{plain}
\newtheorem{theorem}{Theorem}[section]
\newtheorem{lemma}[theorem]{Lemma}

\newtheorem{corollary}[theorem]{Corollary}

\theoremstyle{definition}

\newcommand\xqed[1]{%
    \leavevmode\unskip\penalty9999 \hbox{}\nobreak\hfill
    \quad\hbox{#1}}

\newtheorem{exampleqqq}[theorem]{Example}
\newenvironment{example}{\begin{exampleqqq}}
  {\xqed{$\triangle$}\end{exampleqqq}}

\newtheorem{remarkqqq}[theorem]{Remark}
\newenvironment{remark}{\begin{remarkqqq}}
  {\xqed{$\triangle$}\end{remarkqqq}}

\newtheorem{definition}[theorem]{Definition}

\theoremstyle{remark}

\newenvironment{ack}{\section*{Acknowledgement}}{}

\newcounter{dummy}
\makeatletter
\newcommand\myitem[1][]{\item[#1]\refstepcounter{dummy}\def\@currentlabel{#1}}
\makeatother

\newenvironment{romenumerate}[1][-10pt]{
\addtolength{\leftmargini}{#1}\begin{enumerate}
 \renewcommand{\labelenumi}{\textup{(\roman{enumi})}}%
 }{\end{enumerate}}

\newcounter{oldenumi}
{\setcounter{oldenumi}{\value{enumi}}
\begin{romenumerate} \setcounter{enumi}{\value{oldenumi}}}
{\end{romenumerate}}

\newcounter{thmenumerate}

\newcounter{xenumerate}   

\newcommand\pfitemx[1]{\par#1:}
\newcommand\pfitemref[1]{\pfitemx{\ref{#1}}}

\newcommand{\refT}[1]{Theorem~\ref{#1}}
\newcommand{\refTs}[1]{Theorems~\ref{#1}}

\newcommand{\refL}[1]{Lemma~\ref{#1}}

\newcommand{\refR}[1]{Remark~\ref{#1}}

\newcommand{\refS}[1]{Section~\ref{#1}}

\newcommand{\refD}[1]{Definition~\ref{#1}}

\begingroup
  \divide\count255 by 60
  \multiply\count255 by -60
  \advance\count255 by \time
  \ifnum \count255 < 10 \xdef\klockan{\the\count1.0\the\count255}
  \else\xdef\klockan{\the\count1.\the\count255}\fi
\endgroup

\newcommand{\sumko}{\sum_{k=0}^\infty}

\newcommand{\sumno}{\sum_{n=0}^\infty}

\newcommand{\sumn}{\sum_{n=1}^\infty}
\newcommand{\sumin}{\sum_{i=1}^n}

\newcommand\set[1]{\ensuremath{\{#1\}}}
\newcommand\bigset[1]{\ensuremath{\bigl\{#1\bigr\}}}

\newcommand\lrset[1]{\ensuremath{\left\{#1\right\}}}

\newcommand\bigpar[1]{\bigl(#1\bigr)}
\newcommand\Bigpar[1]{\Bigl(#1\Bigr)}

\newcommand\lrpar[1]{\left(#1\right)}
\newcommand\bigsqpar[1]{\bigl[#1\bigr]}
\newcommand\sqpar[1]{[#1]}

\newcommand\lrsqpar[1]{\left[#1\right]}
\newcommand\cpar[1]{\{#1\}}

\newcommand\abs[1]{\lvert#1\rvert}
\newcommand\bigabs[1]{\bigl\lvert#1\bigr\rvert}

\newcommand\lrabs[1]{\left\lvert#1\right\rvert}
\def\rompar(#1){\textup(#1\textup)}    
\newcommand\xfrac[2]{#1/#2}

\def\xexp(#1){e^{#1}}

\newcommand\ntoo{\ensuremath{{n\to\infty}}}

\newcommand\ttoo{\ensuremath{{t\to\infty}}}
\newcommand\ctoo{\ensuremath{{c\to\infty}}}
\newcommand\bmin{\land}
\newcommand\bmax{\lor}

\newcommand\punkt{\xperiod}    
\newcommand\iid{i.i.d\punkt}    
\newcommand\ie{i.e\punkt}
\newcommand\eg{e.g\punkt}

\newcommand\cf{cf\punkt}
\newcommand{\as}{a.s\punkt}
\newcommand{\aex}{a.e\punkt}

\newcommand{\tend}{\longrightarrow}
\newcommand\dto{\overset{\mathrm{d}}{\tend}}
\newcommand\pto{\overset{\mathrm{p}}{\tend}}
\newcommand\asto{\overset{\mathrm{a.s.}}{\tend}}

\newcommand\eqd{\overset{\mathrm{d}}{=}}

\newcommand\bbR{\mathbb R}

\newcommand\bbN{\mathbb N}

\newcounter{CC}
\newcounter{cc}

\newcommand\E{\operatorname{\mathbb E}{}} 
\renewcommand\P{\operatorname{\mathbb P{}}}

\newcommand\Var{\operatorname{Var}}

\newcommand\ga{\alpha}
\newcommand\gb{\beta}
\newcommand\gd{\delta}

\newcommand\gl{\lambda}

\newcommand\gss{\sigma^2}

\newcommand\eps{\varepsilon}
\renewcommand\phi{\xxx}  

\newcommand\cA{\mathcal A}

\newcommand\indic[1]{\boldsymbol1\cpar{#1}}

\newcommand\qq{^{1/2}}
\newcommand\qqw{^{-1/2}}

\newcommand\intoo{\int_0^\infty}

\newcommand\ooo{[0,\infty)}

\newcommand\dd{\,\mathrm{d}}

\newcommand{\ui}{uniformly integrable}

\newcommand\xoo{_1^\infty}
\newcommand\xooo{_0^\infty}

\newcommand\re{_{r,\eps}}
\newcommand\EE{\E}
\newcommand\Nt{{N(t)}}
\newcommand\Nti{{N(t)+1}}
\newcommand\taut{{\tau(t)}}
\newcommand\hM{\widehat M}

\newcommand\hZ{\widehat Z}
\newcommand\heta{\widehat \eta}
\newcommand\cadlag{c\`adl\`ag}
\newcommand\ac{c}

\newcommand{\Holder}{H\"older}

\begin{document}

\begin{abstract} 
Serfozo (2009, Theorem 2.65) gives a useful central limit theorem 
for processes with regenerative increments.
Unfortunately, there is a gap in the proof.
We fill this gap, and at the same time we weaken the assumptions.
Furthermore, we give conditions for moment convergence in this setting.
We give also further results complementing results in Serfozo (2009)
on the law of large numbers and estimates for the mean;
in particular, we show that there is a gap between conditions for the weak
and strong laws of large numbers.
\end{abstract}

\maketitle

\section{Introduction}\label{S:intro}

The standard setting in renewal theory is that we have a stochastic process
(in continuous or discrete time) such that some event occurs at 
random times $0<T_1<T_2<\dots$, and the process ``starts again'' at each
such event.
Formally this means that the times between renewals 
$\xi_n:=T_n-T_{n-1}$, $n\ge1$, 
(where we define $T_0:=0$)
are \iid{}
(independent and identically distributed) random variables.
We thus have
\begin{align}\label{a1}
  T_n:=\sumin \xi_i,
\qquad n\ge0.
\end{align}
with $(\xi_i)\xoo$ \iid, 
where $\xi_1$ may be any (strictly) positive random variable.
We then define, for $t\ge0$,
\begin{align}\label{a2}
  N(t)&:=\max\set{n:T_n\le t}=\sumn\indic{T_n\le t},
\\\label{a3}
\tau(t)&:=\min\set{n:T_n>t}=N(t)+1.
\end{align}
It is well-known that $T_n\to\infty$ \as{} as \ntoo, and thus $N(t)$ and
$\tau(t)$ are well defined for any $t\ge0$.
Note that by the definitions,
\begin{align}\label{a4}
T_\Nt \le t < T_{\Nti}=T_{\tau(t)}.
\end{align}

In applications it is common to study the values of another stochastic
process at the renewal times $T_n$. One common version of this is to
let $(\eta_i)\xoo$ be another sequence of random variables such that
the random vectors $(\xi_i,\eta_i)$, $i\ge1$, are \iid,
and define their partial sums
\begin{align}\label{a5}
  V_n:=\sumin \eta_i;
\end{align}
we then may 
consider $V_{N(t)}$ or $V_{\tau(t)}$, which can be interpreted as the values
of a stochastic process at the time $T_{N(t)}$ or $T_{\tau(t)}$.
Asymptotic results such as a law of large numbers and a central limit theorem
for $V_{N(t)}$ or $V_{\tau(t)}$ are well known, see \eg{}
\cite[Section 4.2, including Remark 4.2.10]{Gut:SRW}.

A version of this is that we are given another 
(real-valued) stochastic process $Z(t)$
defined for all times $t\ge0$ such that $Z(t)$ also ``starts again'' at
each renewal time $T_n$. We are interested in asymptotic results for $Z(t)$
as \ttoo, and by \eqref{a4}, we may under suitable assumptions approximate
$Z(t)$ by $V_{N(t)}$ or $V_{\tau(t)}$ (with $\eta_n$ given below)
and obtain results for $Z(t)$ from
results of the type just mentioned.
To make this formal, \citet[Definition 2.52, Section 2.10]{Serfozo} 
makes the following definition:
\begin{definition}\label{D1}
   The process $Z(t)$ has
\emph{regenerative increments over the times $T_n$} if
$Z(0)=0$ and
the increments
\begin{align}\label{z1}
  \zeta_n:=\bigpar{\xi_n,\,\set{Z(t+T_{n-1})-Z(T_{n-1}): 0\le t\le \xi_n}},
\qquad n\ge1,
\end{align}
are i.i.d.
\end{definition}

\begin{remark}
Note that the second component of $\zeta_n$ is a stochastic process, defined 
over the random interval $[0,\xi_n]$; we may for the purpose of this
definition regard it as a process on $\ooo$ stopped at $\xi_n$, 
\ie, equal to $Z(\xi_n+T_{n-1})-Z(T_{n-1})$ for all $t\ge\xi_n$. 
In general we may regard this process as an
element of the product space $\bbR^{\ooo} $
but we will
also need some regularity property.
We assume, for convenience, that $Z(t)$ is \cadlag{}
(right-continuous with left limits, also written $Z\in D\ooo$); 
then the second component of
$\zeta_n$ is also \cadlag.
\end{remark}

\begin{remark}
The definition given in  \cite[Definition 2.52]{Serfozo} actually differs
slightly from the one above, using only the interval $0\le t<\xi_n$
instead of $0\le t\le\xi_n$ in \eqref{z1}.
This is obviously a typo, and should be interpreted as in \eqref{z1}, since
otherwise the definition would be trivially satisfied by 
the process $Z(t):=X_{N(t)}$ for \emph{any} random sequence $(X_n)\xooo$
with $X_0=0$;
this evidently cannot imply any limit results.
\end{remark}

Suppose that $Z(t)$ has regenerative increments over $T_n$, and define
\begin{align}\label{eta}
  \eta_n:=Z(T_n)-Z(T_{n-1}),
\qquad n\ge1.
\end{align}
It follows from the definition above, taking $t=\xi_n$ in \eqref{z1}, that 
the sequence of pairs ${(\xi_n,\eta_n)}$, $n\ge1$, is i.i.d.
Hence, we may define $V_n$ by \eqref{a5}, and note that \eqref{eta} yields
\begin{align}\label{z2}
  V_n = Z(T_n).
\end{align}
Consequently, using \eqref{a4}, we have for any $t\ge0$
\begin{align}\label{z3}
  Z(t)= V_{N(t)} + \bigpar{Z(t'+T_{N(t)})-Z(T_{N(t)})},
\end{align}
where $t':=t-T_{N(t)}\in[0,\xi_{N(t)+1})$.
We define, following \cite{Serfozo},
\begin{align}\label{mn}
  M_n:=\sup_{T_{n-1} \le t\le T_n}|Z(t)-Z(T_{n-1})|,
\qquad n\ge1.
\end{align}
In particular,
\begin{align}\label{m1}
M_1:= \sup_{0 \le t\le T_1}|Z(t)|.
\end{align}
Note that it follows from \refD{D1} that the random variables $M_n$ are i.i.d.

It is clear from \eqref{z3} that with suitable conditions on $M_n$,
asymptotic results for $Z(T)$ follow from results for $V_{N(t)}$.
In particular, \citet{Serfozo} 
gives the following central limit theorem,
which is well suited for applications.

\begin{theorem}[{\cite[Theorem 2.65]{Serfozo}}]\label{T1}
Suppose $Z(t)$ is a stochastic process with
regenerative increments over $T_n$ such that $\mu := \E[T_1 ]$, 
$a:= \E[Z(T_1)]/\mu$,
and $\gss:= \Var[Z(T_1 ) - aT_1 ]$ are finite.
In addition, let
$M_1$ be defined by \eqref{m1},
and assume 
$M_1<\infty $ a.s.
Then
\begin{align}\label{t1}
\frac{Z(t)- at}{\sqrt t} \dto N (0, \gss/\mu),  
\qquad \text{as \ttoo}.
\end{align}
\end{theorem}
We include the case $\gss=0$, letting $N(0,0)$ denote (the distribution of) 0.

\begin{remark}
The theorem stated in \cite{Serfozo} also assumes $\E M_1<\infty$, but
the proof below shows that it suffices to assume $M_1<\infty$ a.s., as done
here. 
\end{remark}

Unfortunately, there is a gap in the proof given in \cite{Serfozo}, see
\refR{Rgap}, so we
give a proof filling that gap (under our, weaker, conditions)
in \refS{Spf}.

In \refS{Smom}, we give conditions for moment convergence in 
\refT{T1}.
Furthermore, we give in \refS{SLLN} a weak law of large numbers,
complementing the strong law in \cite{Serfozo}; we show that the weak law
holds under weaker conditions than the strong law.
Finally, \refS{Smean} gives  further  estimates for the mean under various
moment conditions. 

\begin{remark}
  We let throughout the paper the time parameter $t\in\ooo$ be a continuous
variable. Results for a discrete time parameter $t\in\bbN$ follow
immediately, by assuming
that the times $T_n$ are integer-valued and then considering only $t\in\bbN$.
\end{remark}

\subsection{Some notation}\label{SSnot}

We use the notation introduced above throughout the paper.
In particular, $\mu$, $a$, and $\gss$ have the same meanings as in
\refT{T1}.
We also use the \emph{renewal function}
\begin{align}\label{majU}
  U(t):=\sumno \P\bigpar{T_n\le t}=\E N(t)+1.
\end{align}

We let $\dto$, $\pto$, and $\asto$ denote convergence in distribution,
probability, 
and almost surely,
respectively.
Unspecified limits are as \ttoo.

For real $x,y$, we let $x\bmin y:=\min\set{x,y}$
and $x\bmax y:=\max\set{x,y}$.

``Decreasing'' is interpreted in the weak sense.
\section{Proof of \refT{T1}}\label{Spf}

We basically follow the proof in \cite[pp.~136--137]{Serfozo}.
We define as there
\begin{align}\label{b1}
  Z'(t):=\frac{Z(T_\Nt)-aT_\Nt}{\sqrt t}
=t\qqw\sum_{i=1}^{\Nt}\bigpar{\eta_i-a\xi_i},
\end{align}
where we used \eqref{eta} and \eqref{a1},
and note that $X_i:=\eta_i-a\xi_i$ are \iid{} random variables with 
$\E X_i=\E X_1=\E[Z(T_1)]-a\E[T_1]= 0$
and $\Var X_i=\gss$. Hence it follows from 
Anscombe's theorem \cite[Theorem 2.64]{Serfozo}
(see also \cite[Section 1.3]{Gut:SRW}; 
alternatively, use instead Donsker's theorem \cite[Theorem 7.7.13]{Gut}),
together with the (weak) law of large numbers $\Nt/t\pto 1/\mu$,
\cite[Corollary 2.11]{Serfozo}
that
\begin{align}\label{b2}
  Z'(t)\dto N\bigpar{0,\gss/\mu},
\qquad\text{as \ttoo}.
\end{align}
(The case $\gss=0$ is trivial since then $X_i=0$ and thus $Z'(t)=0$ a.s.)

Hence, by Cram\'er--Slutsky's theorem 
\cite[Theorem 5.11.4]{Gut},
it suffices to show that
\begin{align}\label{b3}
 \frac{Z(t)-at}{\sqrt t}-Z'(t) \pto0.
\end{align}
We have, by \eqref{b1},
\begin{align}\label{b4}
  \frac{Z(t)-at}{\sqrt t}-Z'(t) 
=
\frac{Z(t)-Z_\Nt-a(t-T_\Nt)}{\sqrt t}
\end{align}
and thus, recalling \eqref{a4}, \eqref{mn},  and \eqref{a1},
\begin{align}\label{b5}
 \lrabs{ \frac{Z(t)-at}{\sqrt t}-Z'(t)} 
\le
\frac{M_{\Nt+1}+|a|\xi_{\Nt+1}}{\sqrt t}
=:
\frac{Y_{\Nt+1}}{\sqrt t},
\end{align}
where we let
\begin{align}\label{b6}
  Y_n:=M_{n}+|a|\xi_{n}.
\end{align}
Consequently, to show \eqref{b3}, and thus \refT{T1}, it suffices to show
that
\begin{align}  \label{b7}
\frac{ Y_\Nti}{t\qq} \pto0,
\end{align}
or, equivalently,
\begin{align}  \label{b8}
\frac{Y_\Nti}{\Nt\qq} \pto0,
\end{align}
By \eqref{b6}, \eqref{mn} and \refD{D1}, the random vectors $(\xi_n,Y_n)$
are \iid, and thus \eqref{b7} follows from \refL{L1} below (with
$\gd(t):=t\qqw$), which completes the proof of \refT{T1}.
\qed
  
\begin{remark}\label{Rgap}
  The gap in the proof in \cite{Serfozo} is the claim 
$n\qqw Y_n\eqd n\qqw Y_1\asto 0$ as \ntoo{} made there; 
although $Y_n\eqd Y_1$ and $n\qqw Y_1\asto 0$ as \ntoo,
this only shows 
$Y_n/n\qq\pto0$, which in general is not enough to imply \eqref{b8}.
In fact, it is a well-known  consequence 
of the Borel--Cantelli lemmas 
(see \cite[Proposition 6.1.1]{Gut})
that,
for any \iid{} sequence $Y_n$, we have $Y_n/n\qq\asto0$ 
as \ntoo{}
(which does imply \eqref{b8})
if and only if 
$\EE{ Y_1^2}<\infty$, which  requires the stronger assumption 
$\EE{M_1^2}<\infty$
(and also  $\EE{T_1^2}<\infty$ unless $a=0$).
\end{remark}

We used in the proof the following lemma.
(For related results under moment assumptions, see \cite[Theorem
1.8.1]{Gut:SRW}.)
Recall that a family $(X_\ga)_{\ga\in\cA}$ of random variables is \emph{tight}
(a.k.a.\ \emph{stochastically bounded}) if for every $\eps>0$, there exists 
$\ac>0$ such that $\P(|X_\ga|>\ac)<\eps$ for all $\ga\in\cA$.
Recall also that (the distribution of) $\xi_1$ is \emph{arithmetic} if there
exists $d>0$ such that $\xi_1\in d\bbN=\set{d,2d,\dots}$ a.s.; then the
largest such $d$ is called the \emph{span} of $\xi_1$.

\begin{lemma}\label{L1}
With the notations above, 
let $T_n$ be a sequence of renewal times with $\E T_1=\E \xi_1<\infty$,
and let $Y_n$, $n\ge1$, be another sequence of random variables such that
the random vectors $(\xi_n,Y_n)$ are i.i.d.
Then the family of random variables
\set{Y_{N(t)+1}:t\ge0} is tight.
In particular, if
$\gd(t)$ is any positive function such that $\gd(t)\to0$
as \ttoo, then
  \begin{align}\label{l1}
\gd(t)Y_{N(t)+1}\pto 0
\qquad \text{as }\ttoo 
. \end{align}
\end{lemma}

\begin{proof}
By replacing $Y_n$ with $|Y_n|$, we may for convenience assume that $Y_n\ge0$.
Moreover, if $\xi_1$ is arithmetic, with span $d>0$, then it suffices to
consider $t\in d\bbN$. 

Let $\ac>0$.
Then, since $(\xi_{n+1},Y_{n+1})$ is independent of $T_n$ and further
$(\xi_{n+1},Y_{n+1})\eqd (\xi_{1},Y_{1})$,
\begin{align}\label{maj1}
  \P\bigpar{Y_{N(t)+1}>\ac}
&
=\E \sumno \indic{T_n\le t<T_{n+1}, \; Y_{n+1}>\ac}
\notag\\&
=\E \sumno \P\bigpar{T_n\le t<T_{n+1}, \; Y_{n+1}>\ac\mid T_n}
\notag\\&
=\E \sumno \P\bigpar{0 \le t-T_n< \xi_{n+1}, \; Y_{n+1}>\ac\mid T_n}
\notag\\&
=\E \sumno h_\ac(t-T_n),
\end{align}
where
\begin{align}
  h_\ac(s):=\indic{s\ge0}\P\bigpar{\xi_1 > s,\; Y_1>\ac}
=\indic{s\ge0}\P\bigpar{T_1 > s,\; Y_1>\ac}
.\end{align}
Using the renewal function $U(t)$ defined in \eqref{majU},
we can write \eqref{maj1} as
\begin{align}\label{maj2}
    \P\bigpar{Y_{N(t)+1}>\ac}=U*h_\ac(t)
:=\intoo h(t-u)\dd U(u)
.\end{align}
Note that $h_\ac(s)\ge0$ and that $h_\ac(s)$ is  decreasing 
on $\ooo$ with
\begin{align}\label{maj3}
  \intoo h_\ac(s)\dd s
\le\intoo \P(\xi_1>s)\dd s = \E \xi_1=\mu<\infty.
\end{align}
Hence, $h_\ac(s)$ is directly Riemann integrable, and thus the key renewal
theorem 
(see \cite[Theorems 2.35--37]{Serfozo} or \cite[Theorem 2.4.3]{Gut:SRW})
yields (in both the arithmetic and non-arithmetic cases)
\begin{align}\label{maj4}
\lim_{\ttoo}
      \P\bigpar{Y_{N(t)+1}>\ac} 
=
\frac{1}{\mu}\intoo h_\ac(s)\dd s
= \frac{1}{\mu}\intoo \P\bigpar{\xi_1 > s,\; Y_1>\ac}\dd s
=:\gl_\ac.
\end{align}
Since $\P\bigpar{\xi_1 > s,\; Y_1>\ac} \le \P\bigpar{\xi_1 > s}$ and
\begin{align}
  \intoo \P\bigpar{\xi_1 > s}\dd s = \E \xi_1<\infty,
\end{align}
dominated convergence yields
\begin{align}\label{maja}
\lim_{\ctoo}\gl_\ac=
\lim_{\ctoo}\frac{1}{\mu}\intoo \P\bigpar{\xi_1 > s,\; Y_1>\ac}\dd s  
=
\frac{1}{\mu}\intoo \lim_{\ctoo}\P\bigpar{\xi_1 > s,\; Y_1>\ac}\dd s  
=0
.\end{align}

Let $\eps>0$. Then \eqref{maja} shows that we may choose $\ac>0$ such that
$\gl_\ac<\eps$, and then \eqref{maj4} shows that for all sufficiently large
$t$, we have
\begin{align}\label{majb}
  \P\bigpar{Y_{N(t)+1}>\ac}<\eps.
\end{align}
This is enough to imply \eqref{l1}, since 
for large $t$ we also have $\ac\gd(t)<\eps$, and consequently
\begin{align}
  \P\bigpar{\gd(t)Y_{N(t)+1}>\eps}
\le
  \P\bigpar{\gd(t)Y_{N(t)+1}>\gd(t)\ac}
=
\P\bigpar{Y_{N(t)+1}>\ac}<\eps.
\end{align}

Moreover, we have shown that for every $\eps>0$, there exists $\ac$ and
$t_0$ such  that \eqref{majb} holds for $t\ge t_0$.
There exists $n_0$ such that $\P(N(t_0)+1>n_0)<\eps/2$,
and we may increase $\ac$ so that $\P(Y_n>\ac) < \eps/(2n_0)$ for all
$n\le n_0$.
Then, for every $t\le t_0$,
\begin{align}
  \P(Y_{N(t)+1}>\ac) \le \P\bigpar{N(t)+1>n_0}
+\sum_{n=1}^{n_0} \P\bigpar{Y_n>\ac} <\eps.
\end{align}
Hence, there exists $\ac$ such that \eqref{majb} holds for all $t\ge0$,
which proves that the family $\set{Y_{N(t)+1}:t\ge0}$ is tight.
\end{proof}

\section{Moment convergence}\label{Smom}

If we assume further moment conditions, we also have convergence of moments
in \refT{T1}.

\begin{theorem}
  \label{Tmom}
Let $r\ge2$, and
assume in addition to the assumptions of \refT{T1} that
$\E[T_1^r]<\infty$ and $\E[M_1^r]<\infty$.
Then
the family of random variables
\begin{align}\label{tmom1}
  \lrset{\lrabs{\frac{Z(t)- at}{\sqrt t}}^r: t\ge1}
\end{align}
is uniformly integrable, and consequently \eqref{t1} holds with convergence
of all moments (absolute and ordinary) of orders $\le r$.
\end{theorem}

\begin{proof}
It will be convenient to use $\tau(t)=N(t)+1$ instead of $N(t)$, since
$\tau(t)$ is a stopping time.  
We thus define, similarly to  \eqref{b1},
\begin{align}\label{d1}
  Z''(t):=\frac{Z(T_\taut)-aT_\taut}{\sqrt t}
=t\qqw\sum_{i=1}^{\taut}\bigpar{\eta_i-a\xi_i}.
\end{align}
Since $|\eta_n|\le M_n$ by \eqref{eta} and \eqref{mn}, 
we have, using also \eqref{b1} and \eqref{b6},
\begin{align}\label{d2}
  |Z''(t)-Z'(t)|
=\lrabs{\frac{\eta_\taut-a\xi_\taut}{\sqrt t}}
\le
\frac{M_{\taut}+|a|\xi_{\taut}}{\sqrt t}
=
\frac{Y_{\taut}}{\sqrt t}
,\end{align}
Thus \eqref{b5} yields
\begin{align}\label{d5}
 \lrabs{ \frac{Z(t)-at}{\sqrt t}-Z''(t)} 
\le
\frac{2Y_{\taut}}{\sqrt t}
\end{align}
and consequently
\begin{align}\label{d13}
 \lrabs{ \frac{Z(t)-at}{\sqrt t}} 
\le
\lrabs{Z''(t)}+
\frac{2Y_{\taut}}{\sqrt t}
.\end{align}

Since $Z(T_\taut)=V_\taut$ by \eqref{z2}, the uniform integrability
of $\bigset{|Z''(t)|^r:t\ge1}$ follows by \cite[Theorem 4.2.3(ii)]{Gut:SRW}
applied to $\sum_{i=1}^\taut\heta_i$ with $\heta_i:=\eta_i-a\xi_i$;
note that $\E |\eta_1|^r \le \E [M_1^r]<\infty$ and thus also 
$\E |\heta_1|^r <\infty$.

Furthermore, recalling $\xi_1=T_1$, the assumptions and \eqref{b6}
yield $\E|Y_1|^r<\infty$. The family
$\set{\tau(t)/t:t\ge1}$ is \ui{} by 
\cite[(2.5.6), see also the more general Theorem 3.7.1]{Gut:SRW}, 
and $\tau(t)$ are stopping times, and thus 
\cite[Theorem 1.8.1]{Gut:SRW} shows that the family
\bigset{|Y_\taut|^r/t:t\ge1} is \ui.
Since $r\ge2$, this implies that also the family
\bigset{\bigpar{|Y_\taut|/\sqrt t}^{r}:t\ge1} is \ui.

The uniform integrability of \eqref{tmom1} now follows by \eqref{d13}.
As is well known, this implies moment convergence in \eqref{t1}, 
see \eg{} \cite[Theorem 5.5.9]{Gut}.
\end{proof}

\begin{corollary}\label{C2}
Suppose $Z(t)$ is a stochastic process with
regenerative increments over $T_n$ 
such that
$\E[T_1^2 ]$
and $\E[M_1^2]$ are finite,
and $\Var\,[Z(T_1 ) - \frac{\EE{Z(T_1)}}{\EE{T_1}} T_1 ]>0$.
Then
\begin{align}\label{c2}
\frac{Z(t)- \E[Z(t)]}{\sqrt {\Var [Z(t)]}} \dto N (0, 1),  
\qquad \text{as \ttoo}.
\end{align}
\end{corollary}

\begin{proof}
Note that 
$\bigabs{Z(T_1)}\le M_1$ by \eqref{m1}, and thus $\E[Z(T_1)^2]<\infty$.
It follows that 
the assumptions  of \refT{T1} hold, and so do the assumptions of \refT{Tmom}
with $r=2$.

Hence, \eqref{t1} holds, with
\begin{align}\label{c3}
\E[Z(t)]& = at + o\bigpar{\sqrt t}  ,
\\\label{c4}
\Var\,[Z(t)]&= \lrpar{\xfrac{\gss}{\mu}+ o(1)}t.
\end{align}
The result \eqref{c2} follows from \eqref{t1} and \eqref{c3}--\eqref{c4} by
Cram\'er--Slutsky's theorem. 
\end{proof}

\section{Law of large numbers}\label{SLLN}

As a complement to the results on asymptotic normality, 
we give both weak and strong laws of large numbers for processes with
regenerative increments. The strong law is given in \cite{Serfozo}, but
repeated here for completeness.

\begin{theorem}
\label{TLLN}
Suppose $Z(t)$ is a stochastic process with
regenerative increments over $T_n$ such that $\mu := \E[T_1 ]$ and
$a:= \E[Z(T_1)]/\mu$
are finite.
In addition, let
$M_1$ be defined by \eqref{m1}.
\begin{romenumerate}
  
\item\label{TLLN1} (Weak LLN.)
If\/ $M_1<\infty $ a.s., then
$Z(t)/t\pto a$ as \ttoo.

\item\label{TLLN2} (Strong LLN {\cite[Theorem 2.54]{Serfozo}}.)
If\/ $\E M_1<\infty $, then
$Z(t)/t\asto a$ as \ttoo.
\end{romenumerate}
\end{theorem}

\begin{proof}
By \cite[Theorem 4.2.1]{Gut:SRW}, we have the strong law of large numbers
\begin{align}
  \frac{Z(T_\taut)}{t}
=   \frac{V_\taut}{t}
\asto \frac{\E[Z(T_1)]}{\E[T_1]}=a.
\end{align}
Moreover,
$\taut/t\asto1/\mu$ by 
\cite[Theorem 2.5.1(i)]{Gut:SRW},
and thus 
\cite[Theorem 1.2.3(i)]{Gut:SRW} (with $r=1$)
shows that
\begin{align}
  \frac{\eta_\taut}{t}
\asto 0.
\end{align}
Consequently,
\begin{align}
  \frac{Z(T_\Nt)}{t}
=
  \frac{Z(T_\taut)-\eta_\taut}{t}
\asto a.
\end{align}
Hence, the weak and strong laws in \ref{TLLN1} and \ref{TLLN2} are
equivalent to, respectively,
\begin{align}\label{ve1}
  \frac{Z(t)-Z(T_\Nt)}{t}& \pto 0,
\\\label{ve2}
  \frac{Z(t)-Z(T_\Nt)}{t}& \asto 0.
\end{align}
The proof is completed as follows.

\pfitemref{TLLN1}
We have, by \eqref{mn} and \eqref{a4},
\begin{align}\label{ve3}
  M_{\Nt+1}=\sup_{T_\Nt\le s \le T_\Nti}\bigabs{Z(s)-Z(T_\Nt)}
\ge \bigabs{Z(t)-Z(T_\Nt)}.
\end{align}
\refL{L1} with $Y_n:=M_n$ shows that $M_{\Nti}/t\pto0$, and thus \eqref{ve1}
follows from \eqref{ve3}.

\pfitemref{TLLN2}
Since $\Nt/t\asto 1/\mu>0$
\cite[Theorem 2.5.1(i)]{Gut:SRW},
\eqref{ve2} is equivalent to
\begin{align}\label{uj1}
    \frac{Z(t)-Z(T_\Nt)}{\Nt}& \asto 0,
\qquad \text{as \ttoo,}
\end{align}
and thus to
\begin{align}\label{uj2}
  \sup_{T_n \le t< T_{n+1}}  \frac{\abs{Z(t)-Z(T_n)}}{n}& \asto 0,
\qquad \text{as \ntoo}
.\end{align}
Define
\begin{align}\label{mn'}
  M_n':=\sup_{T_{n-1} \le t< T_n}|Z(t)-Z(T_{n-1})|,
\qquad n\ge1;
\end{align}
\cf{} \eqref{mn} and note that
\begin{align}\label{mn+}
M_n =  M_n'\bmax|Z(T_n)-Z(T_{n-1})|
.\end{align}
We can write \eqref{uj2} as
\begin{align}\label{ve6}
    \frac{M_{n+1}'}{n}& \asto 0,
\qquad \text{as \ntoo}
.\end{align}
Since the sequence \set{M'_n} is \iid, \eqref{ve6} is equivalent to 
\begin{align}
  \label{ve7}
\E M'_1<\infty, 
\end{align}
see \cite[Proposition 6.1.1]{Gut}.
Furthermore, $\E|Z(T_1)|<\infty$ by assumption, and thus \eqref{mn+} shows
that \eqref{ve7} is equivalent to
\begin{align}\label{ve8}
  \E M_1 = \E\sqpar{M'_1\bmax|Z(T_1)|}<\infty.
\end{align}
The chain of equivalences above shows that \eqref{ve2} is equivalent to
\eqref{ve8}. 
\end{proof}

\begin{remark}\label{RLLN}
  The proof shows that, under the assumption that $\E[T_1]$ and $\E[Z(T_1)]$ are
  finite, the strong law of large numbers $Z(t)/t\asto a$ holds if and only if  
$\E[M_1]<\infty$. The gap between the conditions in \ref{TLLN1} and
\ref{TLLN2} is thus not an artefact of the proof.
\end{remark}

  \section{The mean}\label{Smean}
We add also some further results for the mean.
First, we note the estimate \eqref{c3} obtained above
when $T_1$ and $M_1$ have finite second moments
can be improved.
In fact, \cite[Theorem 2.85]{Serfozo} shows the following.
(The assumptions in \cite{Serfozo} are slightly more general.
Also, \cite{Serfozo} states only the non-arithmetic case,
but the arithmetic case is similar.) 
For completeness, we give a proof later.

\begin{theorem}[Essentially {\cite[Theorem 2.85]{Serfozo}}]\label{TE2}
  Suppose $Z(t)$ is a stochastic process with
regenerative increments over $T_n$ 
such that
$\E[T_1^2 ]$,
$\E[M_1]$, 
and\/ $\E[T_1M_1]$ are finite.
(In particular, this holds if\/
$\E[T_1^2 ]$ and\/
$\E[M_1^2]$
are finite.)
Then
\begin{align}\label{teO}
\E[Z(t)]& = at + O(1)
.\end{align}
More precisely,
\begin{romenumerate}
\item 
  If the distribution of $T_1$ is non-arithmetic, then, as \ttoo,
\begin{align}\label{te2}
\E[Z(t)]& = at + a\frac{\E[T_1^2]}{2\mu}
+\frac{1}{\mu}\E\lrsqpar{\int_0^{T_1}Z(s)\dd s - T_1 Z(T_1)}
+o(1).
\end{align}

\item 
If the distribution of $T_1$ is arithmetic with span $d$, then
\begin{align}\label{ted}
\E[Z(t)]& 
= at + a\frac{\E[T_1^2]}{2\mu}
+\frac{ad}{2}
+\frac{1}{\mu}\E\lrsqpar{{d}\sum_{k=1}^{T_1/d-1}Z(kd) - T_1 Z(T_1)}
+o(1)
\notag\\&
= at + a\frac{\E[T_1^2]}{2\mu}
-\frac{ad}{2}
+\frac{1}{\mu}\E\lrsqpar{{d}\sum_{k=1}^{T_1/d}Z(kd) - T_1 Z(T_1)}
+o(1)
.\end{align}
as \ttoo{} with $t\in d\bbN$.

\end{romenumerate}
\end{theorem}

\begin{remark}\label{RE2}
  In the special case $Z(t):=V_{\Nt}$ for some sequence $V_n$ as in
  \eqref{a5},
\ie, in the case $Z(t)=Z(T_n)$ for $T_{n}\le t< T_{n+1}$,
\eqref{teO} is a special case of 
\cite[Theorem 4.2.4(i) with Remark 4.2.10]{Gut:SRW},
and the proof given there shows also \eqref{te2} and \eqref{ted}.

The even more special case $Z(t):=N(t)$ is classical, see
\cite[Theorem 2.5.2]{Gut:SRW} and \cite[Proposition 2.84]{Serfozo}.
\end{remark}

Under weaker moment assumptions
(where asymptotic normality  does not necessarily hold),
we have the following results. (Proofs are given below.)
First, we assume only finite expectations.

\begin{theorem}\label{TE1}
  Suppose $Z(t)$ is a stochastic process with
regenerative increments over $T_n$ such that $\E[T_1 ]$
and $\E[M_1]$ are finite.
Then
\begin{align}\label{te1}
\E\sqpar{Z(t)}=at+o(t),
\qquad \text{as \ttoo}.
\end{align}
\end{theorem}

More generally, we have the following theorem that ``interpolates'' between
\refTs{TE2} and \ref{TE1}. Note that the case $r=1$ is \refT{TE1} 
(which we have stated separately for emphasis), 
and that \refT{TE2} is a substitute for the 
excluded case $r=0$.

\begin{theorem}\label{TEr}
  Suppose $Z(t)$ is a stochastic process with
regenerative increments over $T_n$.
Let $0<r\le1$.
\begin{romenumerate}  
\item \label{TEr1}
If\/ $\E[T_1^{2-r}]$,
$\E[M_1]$, and $\E[T_1^{1-r}M_1]$ are finite, then 
\begin{align}\label{ter}
\E\sqpar{Z(t)}=at+o(t^r),
\qquad \text{as \ttoo}.
\end{align}
\item \label{TEr2}
In particular,
\eqref{ter} holds if there exist $p\ge 2-r$ and $q\ge1$
such that $\E[T_1^p]$ and $\E[M_1^q]$ are finite, and
\begin{align}\label{pqr}
  p\Bigpar{1-\frac{1}{q}}\ge 1-r.
\end{align}
\end{romenumerate}
\end{theorem}

\begin{remark}\label{RE1}
  In the special case $Z(t):=V_{\Nt}$ as in \refR{RE2},
\eqref{te1} follows from \cite[Theorem 4.2.1 with Remark 4.2.10]{Gut:SRW}.
Moreover, in this case and for $r\in(\frac12,1)$, 
under somewhat stronger moment assumptions,
\eqref{ter} follows from \cite[Theorem 4.2.2 with Remark 4.2.10]{Gut:SRW}.
These theorems in \cite{Gut:SRW} yield also
estimates for higher moments of $Z(t)-at$; we leave it to the reader to
extend those results to general processes with regenerative increments.

The special case $Z(t)=N(t)$ of \eqref{te1} is the elementary renewal
theorem. 
Furthermore, in this special case, \eqref{ter}  was proved in \cite{Tacklind}.
\end{remark}

\begin{proof}[Proof of \refT{TE1}]
Consider $\hZ(t):=Z(t)-at$,
which also is a process with regenerative increments over $T_n$;
we have $\E[\hZ(T_1)]=\E[Z(T_1)]-a\E[T_1]=0$ 
and 
\begin{align}\label{hM}
\hM_1:=\sup_{0 \le t\le T_1}|\hZ(t)| \le M_1 + |a| T_1.
\end{align}
Hence, by replacing $Z(t)$ with $\hZ(t)$, it follows that
we may without loss of generality assume
$\E[Z(T_1)]=0$ and thus $a=0$.
(We could have done so also in earlier proofs, but we preferred to stay
close to \cite{Serfozo}.)

Thus assume $a=0$. Recall that $\tau(t)$ is a stopping time, and note that
$\E[\tau(t)]<\infty$ for every $t\ge0$, 
see \eg{} \cite[Theorem 2.3.1(ii) or Theorem 2.4.1]{Gut:SRW}.
Hence Wald's equation \cite[Theorem 1.5.3(i)]{Gut:SRW},
\cite[Proposition 2.53]{Serfozo} applies
to $Z(T_\taut)=V_\taut$, which yields
\begin{align}\label{e1}
  \E \bigsqpar{Z(T_\taut)} 
= \E[\tau(t)]\cdot\E[Z(T_1)]=0
.\end{align}
Furthermore, it follows from \eqref{a4}, \eqref{mn}, and $\tau(t)=\Nt+1$ that
\begin{align}\label{e2}
  \bigabs{Z(t)-Z(T_\taut)} 
\le   \bigabs{Z(t)-Z(T_\Nt)} +   \bigabs{Z(T_\taut)-Z(T_\Nt)} 
\le 2 M_\taut.
\end{align}
Consequently,
\begin{align}\label{e3}
 \bigabs{ \E[Z(t)]}
=  \bigabs{ \E\bigsqpar{Z(t)-Z(T_\taut)}} 
\le \E \bigabs{Z(t)-Z(T_\taut)} 
\le 2\E{M_\taut}. 
\end{align}
Moreover, $\taut$ are stopping times, with $\taut/t\to1/\mu$ \as{} as \ttoo,
and the random variables \set{\taut/t:t\ge1} are \ui{}, see
\cite[Theorem 2.5.1 and (2.5.6) (or Theorem 3.7.1)]{Gut:SRW}.
Hence, \cite[Theorem 1.8.1]{Gut:SRW} shows that the assumption
$\E[M_1]<\infty$ implies
\begin{align}\label{e4}
  \E\bigsqpar{M_\taut} = o(t),
\end{align}
and the result follows from \eqref{e3}
\end{proof}

\begin{proof}[Proof of \refT{TEr}]
We note first that \ref{TEr2} follows from \ref{TEr1} and \Holder's
inequality. In fact, suppose that the assumptions of \ref{TEr2} hold.
If $q>1$, let $q'$ be the conjugate exponent defined by $1/q'=1-1/q$;
then
\begin{align}
  \E\bigsqpar{T_1^{1-r}M_1}
\le \E\bigsqpar{T_1^{q'(1-r)}}^{1/q'}\E\bigsqpar{M_1^{q}}^{1/q}<\infty,
\end{align}
since \eqref{pqr} says $p/q'\ge1-r$ and thus $p\ge q'(1-r)$.
Hence, the assumptions of \ref{TEr1} hold.
The case $q=1$ occurs by \eqref{pqr} only for $r=1$, and then the result
immediately follows from \ref{TEr1}.

It thus suffices to prove \ref{TEr1}.
As in the proof of \refT{TE1}, we may assume $a=0$. Then \eqref{e3} holds,
and the result follows from the following lemma.
\end{proof}

\begin{lemma}\label{LEr}
As in \refL{L1},
let $T_n$ be a sequence of renewal times 
and let $Y_n$, $n\ge1$, be another sequence of random variables such that
the random vectors $(\xi_n,Y_n)$ are i.i.d.
Let $0<r\le 1$ and  
assume that
$\E [T_1]$, $\E [Y_1]$ and $\E[T_1^{1-r}Y_1]$ are finite.
Then 
  \begin{align}\label{ler}
\E Y_{\taut} = o(t^r)
\qquad \text{as }\ttoo 
. \end{align}  
\end{lemma}
\begin{proof}
  The case $r=1$ follows by \cite[Theorem 1.8.1]{Gut:SRW} as for \eqref{e4}
in the proof of \refT{TE1}. Hence, we may assume $0<r<1$.
We may also assume that $Y_n\ge0$, by otherwise replacing $Y_n$ by $|Y_n|$.

With these assumptions, we argue similarly to the proof of \refL{L1}.
We have
\begin{align}\label{ma1}
  \E \bigsqpar{Y_{\taut}}
&
=\E \sumno \indic{T_n\le t<T_{n+1}}\cdot Y_{n+1}
\notag\\&
=\E \sumno \E\bigsqpar{Y_{n+1}\cdot\indic{T_n\le t<T_{n+1}}\mid T_n}
\notag\\&
=\E \sumno \E\bigsqpar{Y_{n+1}\cdot\indic{0 \le t-T_n< \xi_{n+1}} \mid T_n}
\notag\\&
=\E \sumno h(t-T_n),
\end{align}
where now
\begin{align}\label{ma2}
  h(s):=\indic{s\ge0}\E\bigsqpar{Y_1\indic{\xi_1 > s}}
=\indic{s\ge0}\E\bigsqpar{Y_1\indic{T_1 > s}}
.\end{align}
Define, for $r,\eps>0$,
\begin{align}\label{ma3}
  h\re(s):=(s^{-r}\land\eps)h(s)
.\end{align}
If $t\ge t_\eps:=\eps^{-1/r}$, then $t^{-r}\le s^{-r}\land\eps$ for every
$s\in[0,t]$, and thus \eqref{ma1} implies
\begin{align}\label{ma4}
t^{-r} \E \bigsqpar{Y_{\taut}}&
=
\E \sumno t^{-r}h(t-T_n)
\le \E \sumno h\re(t-T_n)
\notag\\&
=\intoo h\re(t-u)\dd U(u).
\end{align}
By \eqref{ma2}--\eqref{ma3}, $h\re(s)$ is  decreasing on $\ooo$ with,
using Fubini's theorem, 
\begin{align}\label{ma5}
  \intoo h\re(s)\dd s&
\le\intoo s^{-r}h(s)\dd s
=\intoo \E \bigsqpar{Y_1\indic{T_1>s} s^{-r}}\dd s
\notag\\&
= \E\intoo {Y_1 \indic{T_1>s} s^{-r}}\dd s
=\frac{1}{1-r}\E\bigsqpar{Y_1T_1^{1-r}}<\infty.
\end{align}
Hence, $h\re(s)$ is directly Riemann integrable, and thus \eqref{ma4} and 
the key renewal theorem 
yield, 
in the non-arithmetic case,
for every $\eps>0$,
\begin{align}\label{ma6}
\limsup_{\ttoo} t^{-r} \E \bigsqpar{Y_{\taut}}&
\le \lim_\ttoo \intoo h\re(t-u)\dd U(u)
=
\frac{1}{\mu}\intoo h\re(s)\dd s
=:\gl\re.
\end{align}
Moreover, $h\re(s)\to0$ as $\eps\to0$ for every fixed $s$ by \eqref{ma3},
and the inequality $h\re(s)\le s^{-r}h(s)$ together with \eqref{ma5} 
allows us to use the dominated convergence theorem and conclude that
\begin{align}
\lim_{\eps\to0}\gl\re
=  \frac{1}{\mu}\lim_{\eps\to0}\intoo h\re(s)\dd s
=0.
\end{align}
Hence, \eqref{ler} follows from \eqref{ma6}.

The case when $T_1$ is arithmetic is similar; 
if $d$ is the span of $T_1$, then
it suffices to consider $n\in d\bbN$, and
we then have \eqref{ma6} with
\begin{align}
  \gl\re:=\frac{d}{\mu}\sumno h\re(nd).
\end{align}
Again, $\gl\re\to0$ as $\eps\to0$ by dominated convergence, and \eqref{ler}
follows.
\end{proof}

The following example shows that the condition \eqref{pqr} is best possible.
\begin{example}
Let $\ga>1$ and $\gb>0$.
Let the renewal times
$T_n$ be given by \eqref{a1} where 
$\xi_i$ are \iid{} and have 
the Pareto distribution with  $\P[\xi_n>t]=t^{-\ga}$ for $t\ge1$.
(Thus $\E T_1=\E\xi_1<\infty$, since $\ga>1$.)
Let $Z(t)$ be the process
\begin{align}
  Z(t):=\xi_\taut^\gb\indic{t\ge T_\Nt+1}. 
\end{align}
Recall \eqref{a4}, and note that $Z(t)$ is \cadlag{} and that $Z(T_n)=0$ for
all $n$;  consequently $\E[Z(T_1)]=0$ and $a=0$.
Then, for any $t\ge 1$,
\begin{align}
  \P\sqpar{Z(t)>t^\gb}&
= \sumno \P\bigsqpar{T_{n}+1\le t <T_{n+1},\; \xi_{n+1}>t}
\notag\\&
= \sumno \P\sqpar{T_{n}+1\le t}\P\sqpar{\xi_{n+1}>t}
\notag\\&
= \sumno \P\sqpar{T_{n}\le t-1}t^{-\ga}
\notag\\&
= U(t-1)t^{-\ga}
.\end{align}
Since $U(t)/t\to 1/\mu>0$ as \ttoo, and $U(t)\ge U(0)=1$,
we have $U(t-1)>ct$ for 
some $c>0$ and all $t\ge1$. Hence, for $t\ge1$,
\begin{align}
    \P\sqpar{Z(t)>t^\gb}>c t^{1-\ga}
\end{align}
and thus
\begin{align}
    \E\sqpar{Z(t)}>c t^{1+\gb-\ga}.
\end{align}
Consequently, 
\eqref{te1} does \emph{not} hold for $r:=1+\gb-\ga$.

If we are given $p,q\ge1$ and $0<r<1$ such that \eqref{pqr} does not hold, \ie,
\begin{align}\label{za1}
    p\Bigpar{1-\frac{1}{q}}< 1-r,
\end{align}
choose $\gb:=p/q$ and $\ga:=\gb+1-r$.
Note that $\ga>1$ since \eqref{za1} yields
\begin{align}
  p-\gb < 1-r
\end{align}
and thus 
\begin{align}
  \ga > p \ge1.
\end{align}
This also implies $\E[T_1^p]<\infty$.
Furthermore, $M_1\le \xi_1^\gb=T_1^\gb$ and thus
$\E[M_1^q]=\E[T_1^p]<\infty$.
We have seen that \eqref{te1} does not hold.
\end{example}

Finally, we give a proof of \refT{TE2}, since the statement and 
proof in \cite{Serfozo} do not explicitly include the arithmetic case.
(We find it illustrative to include both cases in our proof. The ideas are
similar to the proof in \cite{Serfozo}, although the details differ.)

\begin{proof}[Proof of \refT{TE2}]
 First, the two expressions in \eqref{ted} are equal, since 
their difference is
$ad-\frac{1}\mu \E\sqpar{d Z(T_1)}=ad-da=0$.

Next, 
note that in the case $Z(t)=at$, simple calculations show that \eqref{te2}
and \eqref{ted} hold (without the remainder term).
Hence, we may again replace $Z(t)$ by $Z(t)-at$ and thus assume that $a=0$.

We then have $\E[Z(T_\taut)]=0$ by \eqref{e1}. We argue as in
\eqref{ma1} and obtain, 
recalling that $\zeta_{n+1}$ in \eqref{z1} is independent of $T_n$
and noting that absolute convergence holds by \eqref{e2} and
\eqref{ma1} (with $Y_n:=M_n$),
\begin{align}\label{mb1}
  \E \sqpar{Z(t)}&
=   \E \bigsqpar{Z(t)-Z(T_{\taut})}
=\E \sumno \indic{T_n\le t<T_{n+1}}\cdot \bigpar{Z(t)-Z(T_{n+1})}
\notag\\&
=\E \sumno \E\bigsqpar{\bigpar{Z(t)-Z(T_{n+1})}\cdot\indic{T_n\le t<T_{n+1}}\mid T_n}
\notag\\&
=\E \sumno \E\bigsqpar{\bigpar{Z(t)-Z(T_{n+1})}\cdot\indic{0 \le t-T_n< \xi_{n+1}} \mid T_n}
\notag\\&
=\E \sumno g(t-T_n),
\end{align}
where
\begin{align}\label{mb2}
  g(s):=\indic{s\ge0}\E\bigsqpar{\bigpar{Z(s)-Z(T_1)}\indic{T_1 > s}}
.\end{align}
We have by \eqref{e2}, for $s\ge0$,
\begin{align}\label{mb3}
  |g(s)|\le \E\bigsqpar{\bigpar{Z(s)-Z(T_1)}\indic{T_1 > s}}
\le \E\bigsqpar{2M_1\indic{T_1 > s}}=2h(s),
\end{align}
where we let $h(s)$ be as in \eqref{ma2} with $Y_1:=M_1$.
Then $h(s)$ is decreasing on $\ooo$ and 
\begin{align}
  \label{mb3b}
\intoo h(s)\dd s=\E[M_1T_1]<\infty
\end{align}
by the calculation in \eqref{ma5} with
$r=0$, and thus $h(s)$ is directly Riemann integrable.
Furthermore, since $Z(s)$ is assumed to be \cadlag, it follows from \eqref{mb2},
using \eqref{e2} and dominated convergence, that $g(s)$ also is \cadlag, and
in particular \aex{} continuous. Hence, using \eqref{mb3},
$g(s)$ too is directly Riemann
integrable, see \cite[Proposition 2.88(c)]{Serfozo}.
In the non-arithmetic case, 
\eqref{mb1}--\eqref{mb2} and
the key renewal theorem now yield, using Fubini's theorem justified by
\eqref{mb3} and \eqref{mb3b},
\begin{align}\label{mb4}
    \E \sqpar{Z(t)}&
\to \frac{1}{\mu}\intoo g(s)\dd s
=  \frac{1}{\mu}\intoo \E\bigsqpar{\bigpar{Z(s)-Z(T_1)}\indic{T_1 > s}}\dd s
\notag\\&
=\frac{1}{\mu}\E\lrsqpar{\int_0^{T_1}Z(s)\dd s - T_1 Z(T_1)}
.\end{align}
(This also follows by \cite[Theorem 2.45]{Serfozo}, applied to
$X(t):=Z(t)-Z(T_\taut)$.) 
In the arithmetic case, with span $d$, we obtain instead,
as \ttoo{} with $t\in d\bbN$,
\begin{align}\label{mb4d}
    \E \sqpar{Z(t)}&
\to \frac{d}{\mu}\sumko g(kd)
=  \frac{d}{\mu}\sumko \E\bigsqpar{\bigpar{Z(kd)-Z(T_1)}\indic{T_1 > kd}}
\notag\\&
=\frac{1}{\mu}\E\lrsqpar{d\sum_{k=0}^{T_1/d-1}Z(kd) - T_1 Z(T_1)}
.\end{align}
Since we have assumed $a=0$, 
these results \eqref{mb4} and \eqref{mb4d} show \eqref{te2} and \eqref{ted},
which completes the proof.
\end{proof}

\begin{ack}
I thank Fabian Burghart and Paul Th\'evenin
for inspiring this work by their proofs in \cite{Fabian&}.
\end{ack}

\newcommand\AAP{\emph{Adv. Appl. Probab.} }
\newcommand\JAP{\emph{J. Appl. Probab.} }
\newcommand\JAMS{\emph{J. \AMS} }
\newcommand\MAMS{\emph{Memoirs \AMS} }
\newcommand\PAMS{\emph{Proc. \AMS} }
\newcommand\TAMS{\emph{Trans. \AMS} }
\newcommand\AnnMS{\emph{Ann. Math. Statist.} }
\newcommand\AnnPr{\emph{Ann. Probab.} }
\newcommand\CPC{\emph{Combin. Probab. Comput.} }
\newcommand\JMAA{\emph{J. Math. Anal. Appl.} }
\newcommand\RSA{\emph{Random Structures Algorithms} }
\newcommand\DMTCS{\jour{Discr. Math. Theor. Comput. Sci.} }

\newcommand\AMS{Amer. Math. Soc.}
\newcommand\Springer{Springer-Verlag}
\newcommand\Wiley{Wiley}

\newcommand\vol{\textbf}
\newcommand\jour{\emph}
\newcommand\book{\emph}
\newcommand\inbook{\emph}
\def\no#1#2,{\unskip#2, no. #1,} 
\newcommand\toappear{\unskip, to appear}

\newcommand\arxiv[1]{\texttt{arXiv}:#1}
\newcommand\arXiv{\arxiv}

\newcommand\xand{and }
\renewcommand\xand{\& }

\def\nobibitem#1\par{}

\end{document}